\documentclass[reqno]{amsart}
\usepackage{amssymb,latexsym,amsmath,amsthm,enumerate,amsbsy}
\usepackage[mathscr]{eucal}
\usepackage{framed,color,graphicx}
\usepackage{mathrsfs}
\usepackage{cite}
\usepackage[all]{xy}
\usepackage{tikz}

\makeatletter
\@namedef{subjclassname@2020}{\textup{2020} Mathematics Subject Classification}
\makeatother

\usetikzlibrary{positioning,decorations.pathreplacing,patterns,decorations.pathmorphing}
\tikzset{%
element/.style={draw, shape=circle, fill=white, inner sep=1.4pt}
}

\DeclareSymbolFont{bbold}{U}{bbold}{m}{n}
\DeclareSymbolFontAlphabet{\mathbbold}{bbold}

\theoremstyle{plain}
\newtheorem{theorem}{Theorem}[section]
\newtheorem{lemma}[theorem]{Lemma}
\newtheorem{corollary}[theorem]{Corollary}
\newtheorem{proposition}[theorem]{Proposition}

\newtheorem{problem}[theorem]{Problem}

\theoremstyle{definition}

\newcommand{\ba}{\mathbf{a}}

\newcommand{\bt}{\mathbf{t}}
\newcommand{\bu}{\mathbf{u}}
\newcommand{\bv}{\mathbf{v}}

\begin{document}

\title[Power semirings of finite groups]
{The finite basis problem for the power semirings of finite groups}

\author{Zidong Gao}
\address{School of Mathematics, Northwest University, Xi'an, 710127, Shaanxi, P.R. China}
\email{zidonggao@yeah.net}

\author{Miaomiao Ren}
\address{School of Mathematics, Northwest University, Xi'an, 710127, Shaanxi, P.R. China}
\email{miaomiaoren@yeah.net}

\author{Xiaolei Shao}
\address{School of Mathematics, Northwest University, Xi'an, 710127, Shaanxi, P.R. China}
\email{xiaoleishao@yeah.net}

\author{Mengya Yue}
\address{School of Mathematics, Northwest University, Xi'an, 710127, Shaanxi, P.R. China}
\email{myayue@yeah.net}

\subjclass[2020]{16Y60, 03C05}
\keywords{Power semiring, finite group, finite basis problem, additively idempotent semiring, variety.}

\maketitle

\begin{abstract}
For any group $G$,
the set of all nonempty subsets of $G$ forms an additively idempotent semiring
under set-theoretic union and elementwise multiplication,
called the power semiring of $G$ and denoted by $\mathcal{P}(G)$.
We prove that for a finite group $G$,
$\mathcal{P}(G)$ has no finite basis for its identities if and only if $|G| \geq 3$.
This completes the classification of the power semirings of finite groups with respect to the finite basis property.
\end{abstract}

\section{Introduction}
A \emph{variety} is a class of algebras that is closed under taking subalgebras,
homomorphic images, and arbitrary direct products. By Birkhoff's celebrated theorem, a class
of algebras is a variety if and only if it is an \emph{equational class};
that is, the class of all algebras satisfying a certain set of identities;
such a set is  an \emph{equational basis} of the variety.
A variety is \emph{finitely based} if it admits a finite equational basis;
otherwise, it is \emph{nonfinitely based}.

An algebra $A$ is finitely based (nonfinitely based) if the variety $\mathsf{V}(A)$ it generates is finitely based or not.
The finite basis problem for a class of algebras, one of the central problems in universal algebra,
concerns the classification of its members with respect to the finite basis property.
The present paper focuses on the finite basis problem for a class of additively idempotent semirings:
the power semirings of finite groups.

An \emph{additively idempotent semiring} (or \emph{ai-semiring} for short) is an algebra \((S, +, \cdot)\)
such that the additive reduct \((S, +)\) is an idempotent commutative semigroup,
the multiplicative reduct \((S, \cdot)\) is a semigroup, and the distributive laws hold:
\[
x(y+z) \approx xy + xz, \quad (x+y)z \approx xz + yz.
\]
Distributive lattices \cite{mmt}, max-plus algebras \cite{aei},
and the flat extensions of groups \cite{jac:flat} are all examples of ai-semirings.
Such algebras have found significant applications in branches of mathematics and computer science,
including tropical geometry~\cite{ms}, theoretical computer science~\cite{go}, and information science~\cite{gl}.

Let $S$ be an ai-semiring. The relation $\leq$ on $S$ defined by
\[
a \leq b \Leftrightarrow a+b=b
\]
is a partial order; indeed, $(S, \leq)$ is an upper semilattice, with the supremum of two elements $a$ and $b$ given by $a+b$.
Moreover, the order $\leq$ is readily seen to be compatible with multiplication,
which explains why such an algebra is also called a \emph{semilattice-ordered semigroup}.
Whenever an order on an ai-semiring is mentioned, it always refers to the order defined above.

A natural source of ai-semirings is the powerset construction applied to an arbitrary semigroup.
Namely, for a semigroup $S$, let $\mathcal{P}(S)$ denote the set of all nonempty subsets of $S$.
Then $\mathcal{P}(S)$ becomes an ai-semiring, called the \emph{power semiring} of $S$~\cite{zhao02},
whose addition and multiplication are defined by
\[
A+B = A \cup B, \quad  A\cdot B = \{ab \mid a \in A,\ b \in B\}.
\]
Note that $S$ embeds into the multiplicative reduct of $\mathcal{P}(S)$ via the mapping $s \mapsto \{s\}$.
For detailed background on power semirings, we refer the reader to~\cite{d, d2, dgv, gv}.

Let $S$ be an ai-semiring.
One can construct an ai-semiring $S^0$ from $S$ by adjoining a new element $0$.
The operations on $S^0=S \cup \{0\}$ are defined by:
\[
(\forall a\in S\cup \{0\}) \quad a+0=0+a=a,\quad a0=0a=0,
\]
while preserving the original operations on $S$.
Then $S$ is a subsemiring of $S^0$, and
$0$ is both the additive minimum element and the multiplicative zero of $S^0$.

In particular, if the empty set is included in the powerset construction above,
the resulting algebra is isomorphic to $\mathcal{P}(S)^0$,
the ai-semiring obtained from $\mathcal{P}(S)$ by adjoining a new element as above.
Some authors adopt this convention and define the power semiring
as the algebra of all subsets of $S$ (including the empty set), often denoted simply by $\mathcal{P}(S)$.
In this paper, however, we reserve $\mathcal{P}(S)$ for the nonempty subsets
and write $\mathcal{P}(S)^0$ when the empty set is included.

Although $S$ and $S^0$ differ only by a single adjoined element,
the relationship between their finite basis properties is far from straightforward
and has been investigated in several works \cite{gjrz2, rz, wrz, yrg}.
It is known that the finite basis property is not always preserved from $S$ to $S^0$:
there are examples where $S$ is finitely based but $S^0$ is nonfinitely based (see \cite{yrg}).
Conversely, no example has yet been found where $S$ is nonfinitely based while $S^0$ is finitely based.
Jackson et al.~\cite[Problem 7.4(3)]{jrz} proposed the following problem:

\begin{problem}\label{prob123}
Under what conditions is the $($non$)$finite basis property preserved when passing from an ai-semiring $S$ to $S^0$?
\end{problem}

A variety is \emph{locally finite} if every finitely generated algebra in it is finite.
A finite algebra $A$ is \emph{inherently nonfinitely based} (INFB)
if it is not contained in any finitely based locally finite variety
(equivalently, every locally finite variety containing it is nonfinitely based).
Since every finite algebra generates a locally finite variety,
it follows that every INFB algebra must be nonfinitely based.
Moreover, every finite algebra whose variety contains an INFB algebra is also INFB.

To the best of our knowledge,
the study of the finite basis problem for the power semirings of semigroups was initiated by Dolinka~\cite{d}.
He proved that $\mathcal{P}(S)^0$ is INFB if $S$ is a finite semigroup such that the five-element Brandt semigroup divides $S$,
and also whenever $S$ is INFB (see Corollaries 6.3 and 6.4 of \cite{d}).
He also posed the problem of studying the finite basis problem for $\mathcal{P}(S)^0$,
where $S$ is a finite semigroup,
and raised the question of describing the finite groups $G$ for which $\mathcal{P}(G)^0$ is nonfinitely based~\cite[Problem 6.5]{d}.

Subsequently, he proved in~\cite[Theorem 2]{d2} that, for a finite group $G$,
the multiplicative reduct of $\mathcal{P}(G)^0$ is not INFB
if and only if $G$ is a Dedekind group; that is, every subgroup of $G$ is normal.
By \cite[Theorem 2.11(3)]{jrz},
an ai-semiring is not INFB if its multiplicative reduct is not INFB.
Therefore, when $G$ is a finite Dedekind group, the semiring $\mathcal{P}(G)^0$ itself is not INFB,
and hence neither is its subsemiring $\mathcal{P}(G)$.

A finite ai-semiring is \emph{strongly nonfinitely based} (SNFB)
if every finite ai-semiring whose variety contains it is nonfinitely based.
Every SNFB ai-semiring is itself nonfinitely based, since it contains itself.
We suspect that the converse is false, although no counterexample is currently known.
Every INFB ai-semiring is SNFB, but the converse is not true.
Jackson, Ren, and Zhao~\cite[Theorem 6.1]{jrz}
proved that every finite ai-semiring whose multiplicative reduct contains a nonabelian nilpotent subgroup is SNFB.
In particular, if $G$ is a finite group containing a nonabelian nilpotent subgroup,
then $\mathcal{P}(G)$ is SNFB, and so
both $\mathcal{P}(G)$ and $\mathcal{P}(G)^0$ are nonfinitely based.

Combining the above results,
if $G$ is a finite Dedekind group containing a nonabelian nilpotent subgroup,
then $\mathcal{P}(G)$ is SNFB but not INFB.
Indeed, \cite[Example~6.7]{jrz} shows that $\mathcal{P}(Q_8)$ provides such an example,
where $Q_8$ denotes the quaternion group.
The problem of classifying $\mathcal{P}(S)^0$ for finite semigroups $S$ with respect to the finite basis property was reiterated in~\cite[Problem~7.2]{jrz}.

Subsequently, Gusev and Volkov~\cite[Theorem~1.1, Remark~2]{gv}
proved that $\mathcal{P}(G)^0$ and $\mathcal{P}(G)$ are both nonfinitely based
whenever $G$ is a finite nonabelian solvable group.
Dolinka, Gusev, and Volkov~\cite[Theorem~A]{dgv} further showed that $\mathcal{P}(S)^0$ is nonfinitely based
if $S$ is a finite inverse semigroup such that either $S$ is not Clifford or all subgroups of $S$
are solvable and at least one of them is nonabelian.

Despite these advances, the finite basis problem for $\mathcal{P}(G)^0$ and $\mathcal{P}(G)$ remained open for all finite groups,
even in the abelian case.
In this paper, we completely resolve the finite basis problem for the power semirings $\mathcal{P}(G)$ of finite groups $G$.
Our main result is the following.

\begin{theorem}\label{mainthm}
Let $G$ be a finite group. Then the power semiring $\mathcal{P}(G)$ is nonfinitely based if and only if $|G| \geq 3$.
\end{theorem}

For the proof, we first note that if $|G| < 3$, then $G$ is either a trivial group or a cyclic group of order $2$.
If $G$ is trivial, then $\mathcal{P}(G)$ is trivial and hence finitely based.
If $G$ is cyclic of order $2$, then $\mathcal{P}(G)$ is isomorphic to the three-element ai-semiring $A$,
which is finitely based by \cite[Theorem 7.3]{jac:flat}; its Cayley tables are given in Table~\ref{tbs10} below.
More precisely, by \cite[Theorem 2.9]{rz},
the ai-semiring variety $\mathsf{V}(A)$ is defined by the following three identities:
\[
x^3\approx x,\quad xy\approx yx,\quad x^2+y^2\approx x^2y^2.
\]

\begin{table}[ht]
\centering
\caption{The Cayley tables of $A$} \label{tbs10}
\begin{tabular}{c|ccc}
$+$      &$0$&$a$&$1$\\
\hline
$0$      &$0$&$0$&$0$\\
$a$      &$0$&$a$&$0$\\
$1$      &$0$&$0$&$1$\\
\end{tabular}\qquad\qquad
\begin{tabular}{c|ccc}
$\cdot$  &$0$&$a$&$1$\\
\hline
$0$      &$0$&$0$&$0$\\
$a$      &$0$&$1$&$a$\\
$1$      &$0$&$a$&$1$\\
\end{tabular}
\end{table}

Therefore, to prove Theorem~\ref{mainthm}, it remains to show the converse:
$\mathcal{P}(G)$ is nonfinitely based whenever $|G| \geq 3$.
The rest of the paper is devoted to this task.
The necessary preliminaries are collected in Section~2.
In Section~3, we establish a sufficient condition for an ai-semiring variety to be nonfinitely based,
and apply it to prove the desired converse.

\section{Preliminaries}\label{prelim}
We begin by introducing some notions and notation.
Let $X$ be a countably infinite set of variables, and let $X^+$ denote the free semigroup over $X$.
Due to distributivity, every \emph{ai-semiring term} (or simply a \emph{term}) over $X$
can be expressed as a finite sum of words from $X^+$.
In what follows, terms are denoted by bold lowercase letters $\mathbf{u}, \mathbf{v}, \mathbf{w}, \dots$,
while ordinary lowercase letters $x, y, z, \dots$ stand for variables.
Let $\bu$ be a term. Then $c(\bu)$ denotes the set of variables occurring in $\bu$.

An \emph{ai-semiring identity} (or simply an \emph{identity}) over $X$ is a formal expression of the form
\[
\bu\approx \bv,
\]
where $\bu$ and $\bv$ are terms over $X$.
The set $P_f(X^+)$ of all nonempty finite subsets of $X^+$ is a subsemiring of the power semiring $\mathcal{P}(X^+)$;
by \cite[Theorem 2.5]{kp}, it is a free ai-semiring.

Let $S$ be an ai-semiring and $\bu\approx \bv$ an identity over $\{x_1,\ldots,x_n\}$.
We say that $S$ \emph{satisfies} $\bu\approx \bv$, or that $\bu\approx \bv$ \emph{holds} in $S$,
if
\[
\bu(a_1, a_2, \dots, a_n) = \bv(a_1, a_2, \dots, a_n)
\]
for all $a_1, a_2, \dots, a_n \in S$,
where $\bu(a_1, a_2, \ldots, a_n)$ denotes the evaluation of $\bu$ in $S$ under the assignment $x_i\mapsto a_i$,
and similarly for $\bv(a_1, a_2, \ldots, a_n)$.
Equivalently, $\varphi(\bu) = \varphi(\bv)$ for every semiring homomorphism $\varphi \colon P_f(X^+) \to S$.
Such a homomorphism is also called an \emph{assignment},
and is uniquely determined by the images of the elements of $X$;
we may denote it simply by $\varphi \colon X \to S$.

For convenience, we write $\bu \preceq \bv$ to denote the identity $\bv \approx \bv+\bu$,
and refer to such an expression as an \emph{ai-semiring inequality} (or simply an \emph{inequality}).
The ai-semiring $S$ satisfies $\bu \preceq \bv$ if and only if
$\varphi(\bu) \leq \varphi(\bv)$ for every assignment $\varphi \colon X \to S$.

Next, we introduce an important class of ai-semirings, namely flat semirings.
By a \emph{flat semiring} we mean an ai-semiring
whose multiplicative reduct has a zero element $0$ and satisfies $a+b=0$ for all distinct $a,b\in S$.
An important class of flat semirings is provided by hypergraph semirings,
which have proved to be a powerful tool in the study of the finite basis problem~\cite{jrz, gjrz, gjrz2}.
Before introducing them, we recall the necessary preliminaries on hypergraphs;
for further details, we refer the reader to \cite{hj,jrz}.

Let $k \geq 3$ be an integer.
A \emph{$k$-uniform hypergraph} $\mathbb{H}$ is a pair $\langle V, E\rangle$,
where $E$ is a family of $k$-element subsets of a set $V$.
Each element of $V$ is a \emph{vertex} of $\mathbb{H}$,
and each element of $E$ is a \emph{hyperedge} of $\mathbb{H}$.

Let $\mathbb{H}$ be a $k$-uniform hypergraph.
A \emph{cycle} of $\mathbb{H}$ is an alternating sequence
\[
v_1, e_1, v_2, e_2, \ldots, v_n, e_n, v_1
\]
of distinct vertices and hyperedges such that $v_1 \in e_1 \cap e_n$ and $v_{i+1} \in e_i \cap e_{i+1}$ for $1 \leq i < n$.
The \emph{length} of this cycle is $n$.
The \emph{girth} of $\mathbb{H}$, denoted by $g(\mathbb{H})$, is the length of its shortest cycles.
If $\mathbb{H}$ contains no cycles, then $\mathbb{H}$ is called a \emph{hyperforest}, and we set $g(\mathbb{H}) = \infty$.

Let $n\geq 2$ be an integer.
Then $\mathbb{H}$ is \emph{$n$-colourable}
if there exists a mapping $\varphi: V \to \{1, 2, \ldots, n\}$ such that $|\varphi(e)| \geq 2$ for all $e \in E$
(i.e., no hyperedge is monochromatic).
The smallest such $n$ for which $\mathbb{H}$ is $n$-colourable is called the \emph{chromatic number} of $\mathbb{H}$, denoted by $\chi(\mathbb{H})$.

The following result is due to Erd\H{o}s and Hajnal~\cite{eh}; see Theorems 2.6 and 2.7 of~\cite{hj} for further discussion.

\begin{lemma}\label{kml}
For any integers $k,m,\ell \geq 2$, there exists a $k$-uniform hypergraph $\mathbb{H}$ such that $g(\mathbb{H}) > \ell$ and $\chi(\mathbb H)>m$.
\end{lemma}

Now let $\mathbb{H}=\langle V, E\rangle$ be a $k$-uniform hypergraph,
with no isolated vertices (i.e., each vertex belongs to at least one hyperedge), and satisfying $g(\mathbb{H}) \geq 4$.
Notice that the condition $g(\mathbb{H}) \geq 4$ ensures that any two distinct hyperedges of $\mathbb{H}$ intersect in at most one vertex,
and that for all $1<\ell \leq k$, a set $\{v_1,\ldots,v_{\ell}\}$ is a subhyperedge (that is, a subset of some hyperedge) if and only if every $2$-element subset of $\{v_1,\ldots,v_{\ell}\}$ is a subhyperedge (see~\cite[Lemma 3.2]{jrz}).

Finally, we introduce two algebraic objects associated with $\mathbb{H}$: the hypergraph semiring $M_{\mathbb{H}}$,
and the associated term $\bt_{\mathbb{H}}$.

A \emph{hypergraph semiring} $M_{\mathbb{H}}$ defined by $\mathbb{H}$ is a flat semiring generated by a copy $\{\mathbf{a}_v \mid v \in V\}$ of $V$, along with special elements $0$ and $1$, and subject to the following rules:
\begin{enumerate}[$(1)$]
\item $0$ is the multiplicative zero element;
\item $1$ is the multiplicative identity;
\item $\mathbf{a}_u \mathbf{a}_v = \mathbf{a}_v \mathbf{a}_u$ for all $u, v \in V$;
\item $\mathbf{a}_{u_1} \cdots \mathbf{a}_{u_k} = \mathbf{a}_{v_1} \cdots \mathbf{a}_{v_k}$ whenever $\{u_1, \ldots, u_k\}, \{v_1, \ldots, v_k\} \in E$;
\item $\mathbf{a}_{u_1} \cdots \mathbf{a}_{u_{k-1}} = \mathbf{a}_{v_1} \cdots \mathbf{a}_{v_{k-1}}$ if there exists $v \in V$ such that
    \[\{u_1, \ldots, u_{k-1}, v\}, \{v_1, \ldots, v_{k-1}, v\} \in E.\]
\end{enumerate}
We let $\ba$ denote the common value of the products in item (4).
When $\mathbb{H}$ consists of a single hyperedge, the corresponding hypergraph semiring is denoted by $M_c(a_1\cdots a_k)$.

The term $\bt_{\mathbb H}$ is defined as follows.
Let $\{x_v \mid v \in V\}$ be a set of variables in bijection with $V$, and set
\[
\bt_{\mathbb H} = \sum_{\{v_1,v_2,\ldots,v_k\} \in E} x_{v_1} x_{v_2} \cdots x_{v_k}.
\]
Let $\varphi: \{x_v \mid v \in V\} \to M_{\mathbb{H}}$ be an assignment.
If $\varphi(x_v)=\mathbf{a}_v$ for all $v \in V$, then $\varphi(\bt_{\mathbb H})=\ba$.

\section{Power semirings of finite groups}
In this section, we use the hypergraph semiring approach to
establish a sufficient condition under which an ai-semiring variety is nonfinitely based.
As an application, we show that the power semiring $\mathcal{P}(G)$ is nonfinitely based whenever $|G| \geq 3$.
To this end, we introduce a three-element ai-semiring $S_7$ (see Table~\ref{tbs7} for its Cayley tables),
which will feature in the theorem below.

\begin{table}[ht]
\centering
\caption{The Cayley tables of $S_7$} \label{tbs7}
\begin{tabular}{c|ccc}
$+$      &$0$&$a$&$1$\\
\hline
$0$      &$0$&$0$&$0$\\
$a$      &$0$&$a$&$0$\\
$1$      &$0$&$0$&$1$\\
\end{tabular}\qquad\qquad
\begin{tabular}{c|ccc}
$\cdot$  &$0$&$a$&$1$\\
\hline
$0$      &$0$&$0$&$0$\\
$a$      &$0$&$0$&$a$\\
$1$      &$0$&$a$&$1$\\
\end{tabular}
\end{table}

Jackson et al.~\cite[Corollary~5.1]{jrz} showed that, up to isomorphism,
$S_7$ is the unique nonfinitely based ai-semiring of order at most three.
This distinguishes $S_7$ among all three-element ai-semirings and makes it a particularly interesting object of study.
Despite its seemingly simple appearance, with only three elements,
a commutative multiplication, and just five nonzero entries in its Cayley tables,
the finite basis problem surrounding $S_7$ turns out to be surprisingly rich and nontrivial.
From different perspectives, \cite{jrz, wrz, yrg, gjrz2} have demonstrated that
$S_7$ can transfer the nonfinitely based property to many other finite ai-semirings.
This naturally leads to the following problem, which was first explicitly posed by Gao et al.~\cite{gjrz2}
(a manuscript in preparation, not publicly available); see also \cite[Problem~1.1]{yrg}.

\begin{problem}\label{problem081410}
Is the ai-semiring $S_7$ strongly nonfinitely based?
\end{problem}

We shall use the following observation.

\begin{proposition}\label{pro26081201}
$\mathsf{V}(S_7)=\mathsf{V}(M_c(a_1\cdots a_k))$.
\end{proposition}
\begin{proof}
Since $S_7$ is isomorphic to the subalgebra $\{0, a_1, 1\}$ of $M_c(a_1\cdots a_k)$,
it follows that $\mathsf{V}(S_7)$ is a subvariety of $\mathsf{V}(M_c(a_1\cdots a_k))$.
Conversely, by \cite[Proposition~2.6]{jrz}, $M_c(a_1\cdots a_k)$ is a member of $\mathsf{V}(S_7)$,
so $\mathsf{V}(M_c(a_1\cdots a_k))$ is a subvariety of $\mathsf{V}(S_7)$.
Therefore, $\mathsf{V}(S_7)=\mathsf{V}(M_c(a_1\cdots a_k))$.
\end{proof}

The next result reveals a connection between $S_7$ and the power semirings of groups.

\begin{proposition}\label{pro26081202}
Let $G$ be a group. Then $\mathcal{P}(G)$ contains a copy of $S_7$ if and only if $|G| \geq 3$.
\end{proposition}
\begin{proof}
Suppose that $|G| \geq 3$.
Let $H = G \setminus \{1\}$, where $1$ is the identity of $G$. Then $|H| \geq 2$.
We claim that $H^2=G$.
Indeed, for each $g \in G \setminus \{1\}$, choose $a \in H \setminus \{g\}$.
Then $a^{-1} \in H$, so $1 = a \cdot a^{-1} \in H^2$.
Since $a \neq g$, we have $a^{-1}g \in H$, and so $g = a \cdot (a^{-1}g) \in H^2$.
Thus $H^2=G$.
Therefore, $\mathcal{P}(G)$ contains the subsemiring
$\{\{1\}, H, G\}$, which is isomorphic to $S_7$ under the mapping
\[
\{1\} \mapsto 1,\quad H \mapsto a,\quad G \mapsto 0.
\]
We have shown that $\mathcal{P}(G)$ contains a copy of $S_7$.

Conversely, if $|G|=1$, then $\mathcal{P}(G)$ is trivial, and so it cannot contain a copy of $S_7$.
If $|G|=2$, then $\mathcal{P}(G)$ contains exactly three elements and satisfies the identity $x^3\approx x$,
whereas $S_7$ does not satisfy this identity.
Consequently, $\mathcal{P}(G)$ is not isomorphic to $S_7$.
\end{proof}

The following theorem is a modification of \cite[Theorem 2.2]{gjrz},
which has proved useful in several contexts (see~\cite{aj, gjrz2}).
Let us fix integers $k, m \geq 3$.
For each integer $n \geq 3$,
let $\mathbb{H}_n=\langle V(\mathbb H_n), E(\mathbb H_n)\rangle$ be a $k$-uniform hypergraph
such that $g(\mathbb{H}_n) > k\binom{kn}{2}$ and $\chi(\mathbb{H}_n) > m$;
the existence of such hypergraphs is guaranteed by Lemma~\ref{kml}.
Furthermore,
for each $n \geq 3$, let $\sigma_n$ be an identity that holds in $S_7$ but fails in $M_{\mathbb{H}_n}$,
and let $\mathcal{W}_{k,m}$ denote the ai-semiring variety
defined by the set of identities $\{\sigma_n \mid n \geq 3\}$.
Then $\mathsf{V}(S_7)$ is a subvariety of $\mathcal{W}_{k, m}$.

For ai-semiring varieties $\mathcal{V}_1$ and $\mathcal{V}_2$,
if $\mathcal{V}_1$ is a subvariety of $\mathcal{V}_2$,
we shall use the interval $[\mathcal{V}_1, \mathcal{V}_2]$ to denote
the set of all subvarieties of $\mathcal{V}_2$ that contain $\mathcal{V}_1$.
Since $\mathsf{V}(S_7)$ is a subvariety of $\mathcal{W}_{k, m}$,
the interval $[\mathsf{V}(S_7), \mathcal{W}_{k, m}]$ is well-defined.
We now state the following theorem.

\begin{theorem}\label{thm251117}
Let $k, m \geq 3$ be integers. Then
every variety in the interval $[\mathsf{V}(S_7), \mathcal{W}_{k, m}]$ is nonfinitely based.
\end{theorem}
\begin{proof}
Let $\mathcal{V}$ be an arbitrary variety in the interval $[\mathsf{V}(S_7), \mathcal{W}_{k, m}]$.
Then $\mathcal{V}$ is a subvariety of $\mathcal{W}_{k, m}$ and contains $\mathsf{V}(S_7)$.
Since $\mathcal{W}_{k, m}$ is defined by $\{\sigma_n \mid n \geq 3\}$,
it follows that $\mathcal{V}$ satisfies the identity $\sigma_n$ for all $n \geq 3$.
The proof proceeds by showing that for every $n \geq 3$, the set of all $n$-variable identities of $\mathcal{V}$
does not form an equational basis of $\mathcal{V}$.
To establish this,
it suffices to prove that for every $n \geq 3$, the hypergraph semiring $M_{\mathbb{H}_n}$ does not lie in $\mathcal{V}$,
while every $n$-generated subsemiring of $M_{\mathbb{H}_n}$ does lie in $\mathcal{V}$.

Indeed, let $n \geq 3$ be an integer.
Since the identity $\sigma_n$ holds in $\mathcal{V}$ but fails in $M_{\mathbb{H}_n}$,
it follows that $M_{\mathbb{H}_n}$ is not in $\mathcal{V}$.
On the other hand, from the proof of \cite[Theorem~4.9, Remark~4.10]{jrz},
we know that every $n$-generated subalgebra $T$ of $M_{\mathbb{H}_n}$
lies in the variety $\mathsf{V}(M_c(a_1 \cdots a_k))$.
By Proposition~\ref{pro26081201},
$\mathsf{V}(M_c(a_1 \cdots a_k))=\mathsf{V}(S_7)$, which is a subvariety of $\mathcal{V}$.
Hence $T$ is in $\mathcal{V}$.

Therefore, $\mathcal{V}$ is nonfinitely based.
\end{proof}


\begin{corollary}\label{propinterval}
Let $G$ be a finite group. If $|G| \geq 3$,
then every variety in the interval $[\mathsf{V}(S_7), \mathsf{V}(\mathcal{P}(G))]$ is nonfinitely based.
In particular, the power semiring $\mathcal{P}(G)$ itself is nonfinitely based.
\end{corollary}
\begin{proof}
Assume that $|G| \geq 3$.
By Proposition~\ref{pro26081202},
$\mathsf{V}(S_7)$ is a subvariety of $\mathsf{V}(\mathcal{P}(G))$,
so the interval $[\mathsf{V}(S_7), \mathsf{V}(\mathcal{P}(G))]$ is well-defined.
To apply Theorem~\ref{thm251117}, set $k = m = |G|\geq 3$, and for each $n\geq 3$,
let $\sigma_n$ denote the inequality
\begin{equation}\label{id26081401}
z \preceq z \bt_{\mathbb H_n},
\end{equation}
where $z \notin c(\bt_{\mathbb H_n})$.

We first show that $\sigma_n$ does not hold in $M_{\mathbb H_n}$.
Indeed, let us take the assignment
$\varphi \colon c(\bt_{\mathbb H_n}) \cup \{z\} \to M_{\mathbb H_n}$
defined by $\varphi(x_v) = \ba_v$ and $\varphi(z) = 1$.
Then $\varphi(\bt_{\mathbb H_n})=\ba$, and so
\[
\varphi(z \bt_{\mathbb H_n})=\varphi(z)\varphi(\bt_{\mathbb H_n})=1\ba=\ba.
\]
Since $1 \not\leq \ba$, we obtain that $\varphi(z) \not\leq \varphi(\bt_{\mathbb H_n})$,
which implies that $\sigma_n$ does not hold in $M_{\mathbb H_n}$.

Finally, we show that $\sigma_n$ is satisfied by $\mathcal{P}(G)$.
Indeed, let $\varphi \colon c(\bt_{\mathbb H_n}) \cup \{z\} \to \mathcal{P}(G)$ be an arbitrary assignment.
Then $\varphi(x)$ is a nonempty subset of $G$ for every $x \in c(\bt_{\mathbb H_n}) \cup \{z\}$.
For each $v \in V(\mathbb H_n)$, choose $h_v \in \varphi(x_v)$, and define the mapping
\[
\psi \colon V(\mathbb H_n) \to \{\{h\} \mid h \in G\}, \quad v \mapsto \{h_v\}.
\]
Then $\psi(v)\subseteq \varphi(x_v)$.
Since $\chi(\mathbb H_n)>m = |G| = |\{\{h\} \mid h \in G\}|$, it follows that some hyperedge is monochromatic;
that is, there exists a hyperedge $\{u_1, u_2, \ldots, u_k\} \in E(\mathbb H_n)$ and an element $g \in G$ such that
\[
\{\psi(u_1), \ldots, \psi(u_k)\} = \{\{g\}, \{g\}, \ldots, \{g\}\}
\]
as a multiset. By Lagrange's theorem, $g^{|G|}=1$. Thus
\begin{equation*}
\begin{split}
\{1\}
&= \psi(u_1) \cdots \psi(u_k) \\
&\subseteq \bigcup_{\{v_1, \ldots, v_k\} \in E(\mathbb H_n)} \psi(v_1) \cdots \psi(v_k) \\
&\subseteq \bigcup_{\{v_1, \ldots, v_k\} \in E(\mathbb H_n)} \varphi(x_{v_1}) \cdots \varphi(x_{v_k})
= \varphi(\bt_{\mathbb H_n}),
\end{split}
\end{equation*}
and so
\[
\varphi(z) = \varphi(z) \cdot \{1\} \subseteq \varphi(z) \varphi(\bt_{\mathbb H_n})=\varphi(z\bt_{\mathbb H_n}).
\]
Hence $\mathcal{P}(G)$ satisfies the inequality $\sigma_n$,
and so $\mathcal{P}(G)$ lies in the variety $\mathcal{W}_{k,m}$.
Consequently, $\mathsf{V}(\mathcal{P}(G))$ is a subvariety of $\mathcal{W}_{k,m}$.

Therefore, by Theorem~\ref{thm251117},
every variety in the interval $[\mathsf{V}(S_7), \mathsf{V}(\mathcal{P}(G))]$ is nonfinitely based.
\end{proof}

Corollary~\ref{propinterval} completes the proof of Theorem~\ref{mainthm}.


\begin{corollary}
Let $\{G_i\}_{i\in I}$ be a finite collection of finite groups containing at least one group of order greater than $2$.
Then every variety in the interval
\[
[\mathsf{V}(S_7),\, \mathsf{V}(\{\mathcal{P}(G_i) \mid i\in I\})]
\]
is nonfinitely based.
In particular, the variety $\mathsf{V}(\{\mathcal{P}(G_i) \mid i\in I\})$ itself is nonfinitely based.
\end{corollary}

\begin{proof}
This is completely analogous to the proof of Corollary~\ref{propinterval};
it suffices to take both $k$ and $m$ to be the least common multiple of the integers $\{|G_i| : i\in I\}$; we omit the details.
\end{proof}

%

\section{Conclusion}
We have completely solved the finite basis problem for the power semirings $\mathcal{P}(G)$ of finite groups $G$:
they are finitely based whenever $|G| \leq 2$, and nonfinitely based otherwise.
This provides a new piece of evidence for Problem~\ref{problem081410}.
The corresponding problem for $\mathcal{P}(G)^0$, however, remains open.

If $|G|=1$, then $\mathcal{P}(G)^0$ is a two-element distributive lattice,
whose variety is defined by the identities
\[
x^2\approx x, \quad xy\approx yx, \quad xy \preceq x
\]
(see~\cite[Table~1]{sr}).
If $|G|=2$, then $\mathcal{P}(G)^0$ is isomorphic to the four-element ai-semiring $A^0$
(where $A$ is the three-element ai-semiring given in Table~\ref{tbs10}),
whose variety is defined by the identities
\[
x^3 \approx x, \quad xy \approx yx
\]
(see~\cite[Lemma~3.7]{rz}).
For $|G|\geq 3$, we have not been able to apply Theorem~\ref{thm251117} to $\mathcal{P}(G)^0$:
the identities \eqref{id26081401} used for $\mathcal{P}(G)$ fail in $\mathcal{P}(G)^0$,
and no suitable alternatives have yet been found.
Whether a different choice of identities could render the theorem applicable remains open.

We are inclined to believe that $\mathcal{P}(G)^0$ is nonfinitely based whenever $|G| \geq 3$.
If it is true, this will contribute to Problem~\ref{prob123}.
If it fails, this will provide the first example of a nonfinitely based ai-semiring $S$
whose extension $S^0$ is finitely based,
and the first example of a nonfinitely based ai-semiring that is not SNFB.
Either outcome would therefore deepen our understanding of the subtle relationship between an ai-semiring $S$ and its extension $S^0$.

\subsection*{Acknowledgment}
Miaomiao Ren, corresponding author, is supported by National Natural Science Foundation of China (12371024, 12571020).
Mengya Yue is supported by the Research Innovation Project for Postgraduates of Northwest University (CX2026052).

\end{document}